\documentclass[11pt]{amsart}

\usepackage{latexsym}
\usepackage{amssymb}
\usepackage{amsmath}
\usepackage{amsfonts}
\usepackage{color}
\usepackage{xcolor}
\usepackage{mathrsfs}

\usepackage{fullpage}

\definecolor{babyblueeyes}{rgb}{0.63, 0.79, 0.95}
\usepackage[backgroundcolor=babyblueeyes,textsize=tiny]{todonotes}
\usepackage{soul}

\usepackage{amscd}

\usepackage{graphicx}
\usepackage{epsfig}

\usepackage{amsthm}
\usepackage[all,cmtip]{xy}
\usepackage[pdfstartview=FitH]{hyperref}
\usepackage{pst-grad} % For gradients
\usepackage{pst-plot} % For axes
\usepackage{comment}
\usepackage{psfrag}
\usepackage{verbatim}
\usepackage{hyperref}
\hypersetup{
colorlinks,
allcolors=blue
}

\usepackage{soul}
\newtheorem{thm}{Theorem}[section]
\newtheorem{cor}[thm]{Corollary}
\newtheorem{lem}[thm]{Lemma}
\newtheorem{prop}[thm]{Proposition}

\newtheorem{theorem}{Theorem}

\theoremstyle{definition}
\newtheorem{definition}[thm]{Definition}

\newtheorem{example}[thm]{Example}
\newtheorem{cexample}[thm]{Counterexample}

\newtheorem{rem}[thm]{Remark}

\newtheorem*{ack}{Acknowledgements}

\numberwithin{equation}{section}

\DeclareMathOperator{\argmin}{argmin}

\DeclareMathOperator{\Hess}{Hess}
\DeclareMathOperator{\arcosh}{arcosh}
\DeclareMathOperator{\arsinh}{arsinh}

\newcommand{\RR}{\mathbb{R}}
\newcommand{\CC}{\mathbb{C}}

\newcommand{\sphere}{\mathrm{\mathbb{S}}}

\newcommand{\Cut}{\mathrm{Cut}}

\newcommand{\inPrb}{{\stackrel{\Prb}{\to}}}

\newcommand{\ZZ}{\mathbb{Z}}

\newcommand{\dd}{\mathsf{d}}

\newcommand{\EE}{\mathbb{E}}
\newcommand{\HH}{\mathbb{H}}
\newcommand{\Prb}{\mathbb{P}}

\begin{document}

%-----------------------------------------------------------------------------
% BEGIN FRONT MATTER ----------------------------------------------------------
%-----------------------------------------------------------------------------

% TITLE

\title[CLT for Fr\'echet means of Riemannian manifolds]{Stability of the cut locus and a Central Limit Theorem for Fr\'echet means of Riemannian manifolds}

\author[B.~Eltzner]{Benjamin Eltzner$^{\sharp,\dagger}$}
\address[B.~Eltzner]{Institut f\"ur Mathematische Stochastik, Georg-August-Universit\"at G\"ottingen Germany}
\email{beltzne@uni-goettingen.de}

\author[F.~Galaz-Garc\'ia]{Fernando Galaz-Garc\'ia$^{*,\dagger}$}
\address[F.~Galaz-Garc\'ia]{Institut f\"ur Algebra und Geometrie, Karlsruher Institut f\"ur Technologie (KIT), Germany}
\email{galazgarcia@kit.edu}

\author[S.~Huckemann]{Stephan Huckemann$^{\sharp,\dagger}$}
\address[S.~Huckemann]{Institut f\"ur Mathematische Stochastik, Georg-August-Universit\"at G\"ottingen Germany}
\email{huckeman@math.uni-goettingen.de}

\author[W.~Tuschmann]{Wilderich Tuschmann$^{*,\dagger}$}
\address[W.~Tuschmann]{Institut f\"ur Algebra und Geometrie, Karlsruher Institut f\"ur Technologie (KIT), Germany}
\email{wilderich.tuschmann@kit.edu}

%SUPPORT ACKNOWLEDGMENT

\thanks{$\sharp$ Supported by the RTG 2088 ``Discovering Structure in Complex Data'' at the University of G\"ottingen and by the DFG HU 1575/7 ``Smeary Limit Theorems''.}
\thanks{$^*$Supported by the DFG Priority Program SPP 2026 ``Geometry at Infinity'' and by the RTG 2229 ``Asymptotic Invariants and Limits of Groups and Spaces'' at KIT/Universit\"at Heidelberg.}
\thanks{$^\dagger$ Supported by HeKKSaGOn German--Japanese University Network, Research Project VIII: ``Mathematics at the Interface of Science and Technology towards Innovation''. }

% DATE
\date{\today}

% MATH SUBJECT CLASSIFICATION AND KEYWORDS


% ABSTRACT

\begin{abstract} We obtain a Central Limit Theorem for closed Riemannian manifolds, clarifying along the way the geometric meaning of some of the hypotheses in Bhattacharya and Lin's Omnibus Central Limit Theorem for Fr\'echet means. We obtain our CLT assuming certain stability hypothesis for the cut locus, which always holds when the manifold is compact but may not be satisfied in the non-compact case. 
\end{abstract}
\maketitle

%-----------------------------------------------------------------------------
% END FRONT MATTER ----------------------------------------------------------
%-----------------------------------------------------------------------------

%-----------------------------------------------------------------------------
% BEGIN MAIN MATTER ----------------------------------------------------------
%-----------------------------------------------------------------------------

%--------------------------------------------------
%		SECTION:TO DO
%--------------------------------------------------

%--------------------------------------------------
% 		SECTION: INTRODUCTION
%--------------------------------------------------

\section{Introduction}

Statistics on Riemannian manifolds has attracted much interest, it being the natural setting for considering data on smooth curved spaces. In this context, Bhattacharya and Patrangenaru proved in their seminal work \cite{BP2003,BP2005} a Central Limit Theorem (CLT) for intrinsic Fr\'echet means on Riemannian manifolds (see \cite[Theorem 2.1]{BP2005}). More recently, Bhattacharya and Lin obtained an Omnibus CLT for intrinsic Fr\'echet means on more general metric spaces and discussed some of its Riemannian consequences (see \cite[Theorem 2.2]{BL} or Section~\ref{S:RCLT} below). In this note we further examine the Omnibus CLT in the Riemannian setting and shed light on the roles played by the cut locus and compactness.

Before stating our main result, let us discuss its context (see Section~\ref{S:RCLT} for a more detailed discussion). Let $M$ be a complete Riemannian $m$-manifold and let $\mu$ be a probability measure on $M$. We denote the distance between two points $p,q\in M$ by $\dd(p,q)$ and let $U_p$ be an open normal neighborhood around $p\in M$, so that there exists a system of normal coordinates $\varphi\colon U_p\subseteq M\to V\subseteq T_pM \cong\RR^m$ given by the inverse of the exponential map $\exp_p\colon T_pM \to M$.  Necessary conditions for the $\sqrt{n}$-asymptotic Gaussian CLT for intrinsic Fr\'echet means on Riemannian manifolds are obviously: 
\begin{itemize}
 \item[(A1)] The existence and uniqueness of the Fr\'echet mean of the measure $\mu$ (which is a difficult issue, e.g. \cite{Ka1977,KWS90,Le98,Gr05,Af} not covered here).
 \item[(A3)] $\mu_n \inPrb \mu$, where $\mu_n$ is the empirical distribution based on $n$ independent random variables with common distribution $\mu$ (which can be assumed due to the very general strong laws of Ziezold \cite{Zi} and Bhattacharya--Patrangenaru \cite{BP2003}).
 \item[(A4)] The existence of moments of derivatives of the function given by $v\mapsto \dd(\phi^{-1}(v),p)^2$, for $p\in M$ (1st moment of 2nd derivative and 2nd moment of 1st derivative), which go into the covariance of the asymptotic Gaussian distribution (which are the analogs of standard requirements in the Euclidean case).
  \item[(A6)] Non-singularity of the Hessian $\Hess_{v=0}\EE[\rho(\varphi^{-1}(v),p)^2]$, for $p \in M$, which is a necessary condition for a $\sqrt{n}$-limiting law.
\end{itemize}
In addition to (mild generalizations of) the preceding conditions (see Section~\ref{S:RCLT}), the Omnibus CLT in \cite{BL} additionally requires: 
\begin{itemize}
 \item[(A2)] Twice differentiability of the function $v\mapsto \dd(\varphi^{-1}(v),p)^2$ for $\mu$-a.e. $p\in M$. 
 \item[(A5)] Locally uniform $L^1$ smoothness of the Hessian of $\dd(\varphi^{-1}(v),p)^2$, for $p\in M$.
\end{itemize}
We believe that (A5) is deeply related to (A2). In particular, it would seem from \cite[Corollary 2.3]{BL} that (A2) is implied by the following condition involving the cut locus of the Fr\'echet mean:
\begin{itemize}
 \item[(C)] There exists a neighborhood of the cut locus of the unique Fr\'echet mean carrying no probability.
\end{itemize}
In this contribution, we show that this is true on compact Riemannian manifolds but false in general, leading to the following refinement of \cite[Corollary 2.3]{BL}. Recall that a manifold is \emph{closed} if it is compact and has no boundary.

\begin{theorem}[\protect{CLT for intrinsic Fr\'echet means on closed Riemannian manifolds}]
\label{T:MAIN_THM}
 Let $(M,g)$ be a complete Riemannian $m$-manifold. Let $\mu$ be a probability measure on $M$ with unique Fr\'echet mean $q_o$. Suppose that the cut locus is topologically stable and that there exists a neighborhood $W$ of $\Cut(q_o)\subseteq M$ such that $\mu(W)=0$. Let $U_{q_o}\subseteq M$ be an open normal neighborhood of the mean $q_o$, so that
	\[
	\varphi=\exp_{q_o}^{-1}\colon U_{q_o}\subseteq M \longrightarrow \varphi(U_{q_o})\subseteq T_{q_o}M\cong \mathbb{R}^m
	\]
is a diffeomorphism. If assumptions (A4)--(A6) hold, then
\begin{align*}
\label{EQ:CLT_THM_RIEM}
	\sqrt{n}(\varphi(q^n_o)-\varphi(q_o)) \stackrel{\mathcal{L}}{\longrightarrow} N(0,\Lambda^{-1}C\Lambda^{-1}) \text{ as $n\to \infty$}.
\end{align*}
\end{theorem}

We obtain Theorem~\ref{T:MAIN_THM} as a consequence of the Omnibus CLT, under the more general assumption that, for any point $p\in M$, the cut locus $\Cut(p)$ of $p$ is \emph{(topologically) stable}, i.e.\ that for any open neighborhood $U(\Cut(p))$ of $\Cut(p)$ in $M$, there exists $r>0$ such that $\Cut(B(p,r))\subseteq U(\Cut(p))$. This property is implied by the compactness of $M$ and is in general not true for complete, non-compact Riemannian manifolds (see Section~\ref{S:CUT_LOCUS}). Example~\ref{EX:CYLINDER_CONT} shows that this stability property is in fact necessary in order to apply the Omnibus CLT. As pointed out in \cite{BL}, Theorem~\ref{T:MAIN_THM} improves upon Bhattacharya and Patrangenaru's CLT for the intrinsic Fr\'echet mean in \cite{BP2005} and \cite[Theorems 2.3 and 5.3]{BB}.

Our article is organized as follows. In Section~\ref{S:FRECHET_FUNCTION_MEAN} we recall basic material on the Fr\'echet function and the Fr\'echet mean. In Section~\ref{S:CUT_LOCUS} we recall basic material on the cut locus, discuss different notions of stability, and show that the cut locus of a closed Riemannian manifold is topologically stable. We use this fact in Section~\ref{S:RCLT}, where we discuss the Omnibus CLT and prove Theorem~\ref{T:MAIN_THM}.

% ACKNOWLEDGEMENTS

\begin{ack}The authors wish to thank Luis Guijarro for helpful conversations.   
\end{ack}

%------------------------------------------------------------------------------
% 		SECTION: BASIC DEFINITIONS AND RESULTS
%------------------------------------------------------------------------------
\section{Fr\'echet means}
\label{S:FRECHET_FUNCTION_MEAN}

In this section we recall some basic material on Fr\'echet means.  

\subsection{Basic definitions} Although we will work in the Riemannian setting, the basic objects we consider can be defined in the more general context of \emph{metric measure spaces} and we will define them in this generality.
For an in depth review of means in the Riemannian case, including a historical discussion, we refer the reader to \cite{Af}.

% DEFINITION: METRIC MEASURE SPACE

\begin{definition}
Let $(Q,\dd)$ be a complete, separable metric space and let $\mu$ be a non-negative, locally finite measure on the Borel $\sigma$-algebra (or $\sigma$-field) of $Q$. The triple $(Q,\dd,\mu)$ is a \emph{metric measure space}. %
\end{definition}

Usual examples of metric measure spaces include complete Riemannian $m$-manifolds with the canonical $m$-dimensional volume measure, $m$-dimensional Alexandrov spaces (with curvature bounded below) with the $m$-dimensional Hausdorff measure, $\mathrm{RCD}^*(K,N)$ spaces. Alexandrov spaces generalize complete Riemannian manifolds with a uniform lower sectional curvature bound, while $\mathrm{RCD}^*(K,N)$ spaces generalize complete Riemannian manifolds with Ricci curvature bounded below by $K$ and dimension bounded above by $N$. 

% DEFINITION: FRECHET FUNCTION

\begin{definition} Let $(Q,\dd,\mu)$ be a metric measure space.  The \emph{Fr\'echet function} of the measure $\mu$ is the real-valued function
\begin{align}
	F\colon Q & \to \mathbb{R}\nonumber\\
	q  	& \mapsto \int_{Q} \dd^2(q,p)d\mu(p)
\end{align}
\end{definition}

The value $F(q)\in\mathbb{R}$, if existent, is  the expected squared distance from the point $q\in Q$. 

Suppose that $F$ is finite for some $q\in Q$ and that it has a minimum value
\[
m = \min_{q\in Q}F(q).
\]
Let $F^{-1}(m)\subset Q$ be the preimage of the minimum value $m$. As customary, we call the elements of $F^{-1}(m)$ the \emph{minimum points} (or \emph{minimizers}) of the Fr\'echet function $F$.
 This set is sometimes denoted in the literature by
\[
\argmin_{q\in Q}F(q).
\]
Note that the Fr\'echet function always has a minimum if $Q$ is compact.

% DEFINITION: FRECHET MEAN SET

\begin{definition}
Let $(Q,\dd,\mu)$ be a metric measure space, let $F$ be its Fr\'echet function and assume that $F$ has a minimum value $m$. The set $F^{-1}(m)$ of minimum points of $F$ is  the \emph{Fr\'echet mean set of $(Q,\dd,\mu)$}. 
If the Fr\'echet mean set consists of a single point $q_o\in Q$, we say that \emph{the Fr\'echet mean of  $(Q,\dd,\mu)$ exists and is equal to $q_o$}. 
\end{definition}

Note that in the literature the Fr\'echet mean is also known as the \emph{barycenter} or \emph{center of mass} (cf. \cite{Af,BBI,St}). 
In the special case where the metric measure space is a (complete) Riemannian manifold with the volume measure, the Fr\'echet mean is known as the \emph{(Riemannian) center of mass} or \emph{Karcher mean} (cf. \cite{GK,Ka1977,Ka2014}).

The Fr\'echet mean set exists under fairly general conditions on $(Q,\dd,\mu)$. Indeed, in the Riemannian case, completeness of the metric and finiteness of the Fr\'echet function suffice. This was proven in \cite{BP2003}. The argument, however, is metric and carries over to metric measure spaces that satisfy the Heine--Borel property (i.e.\ all closed bounded subsets are compact).

% PROPOSITION

\begin{prop}[cf.\ \protect{\cite[Theorem 2.1 (a)]{BP2003}}] Let $(Q,\dd,\mu)$ be a metric measure space that satisfies the Heine--Borel property and let $F$ be its Fr\'echet function. If there is a point $q\in Q$ such that that $F(p)$ is finite, then the Fr\'echet mean set is a nonempty compact subset of $Q$.
\end{prop}

Granted the existence of the Fr\'echet mean set, it is of interest to find conditions that ensure that it  consists of a single point. 
This is the case, for example, when the metric measure space is a $\mathrm{CAT}(0)$ space (i.e.\ it is simply connected and has non-positive curvature in the triangle-comparison sense). This was proven for Riemannian manifolds in \cite[Theorem 2.1 (b)]{BP2003} and, for $\mathrm{CAT}(0)$ metric measure spaces, by Sturm \cite{St}. In the case of lower curvature bounds, the situation is more complicated (see, for example, \cite{Oh} for Alexandrov spaces with non-negative curvature or \cite{Af} for Riemannian centers of mass).

% REMARK

\begin{rem}
One may also consider   \emph{extrinsic} Fr\'echet means, defined in terms of an isometric embedding of a 
metric space into an ambient Euclidean space (see \cite[Section 3]{BP2003}). We will not consider these here. The Fr\'echet mean we have defined is also known as the \emph{intrinsic} Fr\'echet mean (cf. \cite{BP2003}). Since this object is an invariant of the metric measure space, we will refer to it, when it exists, simply as the \emph{Fr\'echet mean} (of the metric measure space). 
\end{rem}

% PROBABILISTIC PRELIMINARIES
%------------------------------------------------

\subsection{Fr\'echet sample means} 
Let $(Q,\dd,\mu)$ be a complete metric measure space. Assume that $F$ is finite on $Q$ and has a unique minimizer $q_o$ (which is, by definition, the Fr\'echet mean of $\mu$). Let $n$ be a positive integer and let  $\{Y_j\}_{j=1}^n$ be independent random variables  taking values on $Q$ with common distribution $\mu$. The \emph{empirical distribution $\mu_n$ based on the random variables $Y_j$} is given by
\[
\mu_n = \frac{1}{n}\sum_{j=1}^n\mathbf{1}_{Y_j},
\]
where $\mathbf{1}_{Y_j}$ is the characteristic function of $Y_j$.
The \emph{Fr\'echet sample mean (set)} is the Fr\'echet mean (set) of $\mu_n$, that is, the set of minimizers of the   Fr\'echet function of $\mu_n$, given by
\[
F_n(q)=\frac{1}{n}\sum_{j=1}^n\dd(q,Y_j)^2.
\]
Under broad conditions, for example, for metric spaces which satisfy the Heine--Borel property, every measurable choice $q_o^n$ from the Fr\'echet sample mean set  of $\mu^n$ is a \emph{consistent estimator} of $q_o$,  that is, $q_o^n\to q_o$ almost surely, as $n\to\infty$ (see \cite[Theorem 2.3]{BP2003}).
 More generally, this also holds for separable pseudo-metric spaces, by the work of Ziezold \cite{Zi} (noting that in \cite{Zi} a \emph{pseudo-metric} is called a \emph{finite quasi-metric}). The requirement that the space satisfies the Heine-Borel property has been relaxed in \cite[Theorem A.4]{Hu_Annals_Stat}, where one only requires that the closure of the union from some finite $n$ to $\infty$ of the sample mean sets satisfies the Heine-Borel property. This allows, for example, for infinite dimensional Riemannian manifolds and Alexandrov spaces.

%	REM: REGULARITY OF THE FRECHET FUNCTION
\begin{rem}
The (local) regularity of the Fr\'echet function $F$ on a Riemannian manifold with a probability measure $\mu$ is an important assumption in the proofs of central limit theorems for Fr\'echet means. This is already the case  in Bhattacharya and Patrangenaru's Central Limit Theorem for Fr\'echet means on Riemannian manifolds \cite[Theorem~2.1]{BP2005} and is also assumed as a hypothesis in the proof of the Omnibus Central Limit Theorem (see \cite[Theorem 2.2]{BL} and compare with Theorem~\ref{T:OCLT}). Obtaining necessary and sufficient conditions that ensure the regularity of the Fr\'echet function is a non-trivial problem, as it is important to understand how the cut loci of nearby points are related (see Section~\ref{S:CUT_LOCUS}).
\end{rem}

% CUT LOCUS
%-------------------

\section{The Riemannian cut locus}
\label{S:CUT_LOCUS}

In this section we review some basic facts about the cut locus, following \cite[Chapter III, Section 4]{Sa}. We then consider stability conditions for the cut locus that will play a role in the next section.

\subsection{Background}
Let $M$ be a complete Riemannian manifold, i.e.\ a $C^\infty$-manifold with a complete $C^\infty$ Riemannian metric.  We let $U_pM\subset T_pM$ be the unit tangent space of $M$ at $p$. We will denote the unit tangent bundle of $M$ by $UM$. Given $p\in M$, the exponential map $\exp_p$ is defined on all of $T_pM$ and, since $M$ is complete, there exists a minimal (i.e. distance-realizing geodesic) segment between any two points of $M$. We say that a geodesic $\gamma\colon[a,b]\to M$ is \emph{normal} if it is parametrized with respect to arc-length, i.e. if $||\gamma'(t)||=1$ for all $t\in[a,b]$. Given a unit tangent vector $u\in U_pM$, we will denote by $\gamma_u\colon [0,\infty])\to M$ the geodesic starting out from $p$ with initial direction $u$. We assume throughout that our Riemannian manifolds are connected. 

Define
\begin{align}
t(u) = \sup \{\, t>0 \mid d(p,\gamma_u(t)) = t \,\},
\end{align}
that is, $t(u)$ is the supremum of the values $t$ such that $\gamma_u|_{[0,t]}$ is a minimal geodesic. Clearly, $0<t(u)\leq \infty$ and, if $t(u)$ is finite, then it is the last value of $t$ such that $\gamma_u|_{[0,t]}$ is minimal. 
Recall that if $\gamma$ is a geodesic segment joining two points $p,q\in M$, $q$ is \emph{conjugate to $p$ along $\gamma$} if there is a non-zero Jacobi field along $\gamma$ vanishing at $p$ and $q$. 
It is a basic fact in Riemannian geometry that geodesics emanating from a given point $p\in M$ do not minimize past their first conjugate point.

% PROP

\begin{prop}[\protect{cf.\ \cite[III, Proposition 4.1]{Sa}}] Let $M$ be a complete Riemannian manifold, fix $p\in M$ and let $u\in U_pM$ be a unit tangent vector.
\begin{itemize}
	\item[(1)] Suppose that $t(u)<\infty$. Then $T=t(u)$ if and only if $\gamma_u|_{[0,T]}$ is a normal minimal geodesic and at least one of the following conditions holds:
	\smallskip
	\begin{itemize}
		\item[(a)] $\gamma_u(T)$ is the first conjugate point of $p$ along $\gamma_u$;\smallskip
		\item[(b)] there exists a unit tangent vector $v\in U_pM$, $v\neq u$, such that $\gamma_u(T)=\gamma_v(T)$.
	\end{itemize}
	\item[(2)] The function $t\colon UM\to [0,\infty]$ given by $u\mapsto t(u)$ is continuous. 
\end{itemize}
 \end{prop}

Note that if $M$ is compact, then $t(u)$ is finite for any $u\in UM$, and vice versa. 

\begin{definition} Let $M$ be a complete Riemannian manifold, let $p\in M$ and fix $u\in U_pM$. If $t(u)<\infty$, we call $t(u)u\in T_pM$ the \emph{tangent cut point} of $p$ along $\gamma_u$. Similarly, we call $\exp_pt(u)u$ the \emph{cut point} of $p$ along $\gamma_u$. The sets
\begin{align*}
	\widetilde{\Cut}(p) 	& = \{\, t(u)u \mid u\in U_pM,\ t(u)<\infty\,\},\\
	\Cut(p) 			& = \exp_p(\widetilde{\Cut}(p))
\end{align*}
are called, respectively, the \emph{tangent cut locus}, and the \emph{cut locus} of $p$. Let 
\[
\widetilde{\mathrm{Int}}(p) = \{\, tu \mid 0<t<t(u),\ u\in U_pM \,\}.
\]
We call 
\[
\mathrm{Int}(p) = \exp_p(\widetilde{\mathrm{Int}}(p))
\]
the \emph{interior set} at $p$.
\end{definition}

The tangent cut locus, the cut locus and the interior set of $p$ are related in the following way.
\begin{prop}[\protect{cf.\ \cite[III, Lemma 4.4]{Sa}}] Let $M$ be a complete Riemannian manifold and fix $p\in M$. Then the following assertions hold:
\begin{itemize}
	\item[(1)] $\mathrm{Int}(p)\cap \Cut(p) = \emptyset$, $M=\mathrm{Int}(p)\cup \Cut(p)$, and $\overline{\mathrm{Int}(p)}=M$.\smallskip
	\item[(2)] $\widetilde{\mathrm{Int}}(p)$ is a maximal domain containing the origin $o_p\in T_pM$, on wich $\exp_p$ is a diffeomorphism. \smallskip
	\item[(3)] $\Cut(p)$ has volume zero in $M$, and $\dim(\Cut(p))\leq m-1$.
\end{itemize}
\end{prop}

It follows that, for any point $q\in M\setminus \Cut(p)$, there exists a unique normal minimal geodesic joining $p$ to $q$. In particular, if $M$ is compact and $m$-dimensional, then $\widetilde{\mathrm{Int}}(p)$ is an $m$-dimensional open ball whose boundary $\partial\tilde{\Cut}(p)$ is homeomorphic to $S^{m-1}$. Hence $M$ may be obtained from the cut locus $\Cut(p)$ by attaching an $m$-dimensional ball via the exponential map $\exp_p\colon\widetilde{\Cut}(p)\to\Cut(p)$ 
Moreover, $\Cut(p)$ is a strong deformation retract 
of $M\setminus \{p\}$. Observe that, if $q$ is a cut point of $p$ along $\gamma$, then $p$ is a cut point of $q$ along $t\mapsto \gamma(t(u) - t)$. 

Now we recall some facts about the distance function $\dd\colon M \to \RR$ of a complete Riemannian manifold $M$. Given $p\in M$, we define the \emph{distance function $d_p\colon M\to \mathbb{R}$ to $p$} by 
\[
d_p(q) = \dd(p,q).
\]
The function $d_p$ is a continuous function.

\begin{prop}[\protect{cf.\ \cite[III, Proposition 4.8]{Sa}}] Let $M$ be a complete Riemannian manifold and fix $p\in M$. Then the following assertions hold:
\begin{itemize}
	\item[(1)] The distance function $d_p$ is of class $C^\infty$ on $M\setminus \{\Cut(p)\cup \{p\}\}$.\smallskip
	\item[(2)] The gradient vector $\nabla d_p(q)$ of $d_p$ at $q\in M\setminus \{\Cut(p)\cup \{p\}\}$ is given by 
	\begin{align}
		(\nabla d_p)(q) = \frac{\partial}{\partial t}\gamma_{pq}(d_p(q)),
	\end{align}
	where $\gamma_{pq}$ denotes a unique minimal geodesic from $p$ to $q$ parametrized by arc length. In particular, $||(\nabla d_p)(q)||=1$.\smallskip
	\item[(3)] If there exist at least two normal minimal geodesics $\gamma_1$, $\gamma_2$ joining $p$ to $q$, then $d_p$ is not differentiable at $q$. Note that such $q$ belong to $\Cut(p)$.
\end{itemize}
\end{prop}

For the Hessian of $d_p$ we have the following result, which follows from the second variation formula.

\begin{prop}[\protect{cf.\ \cite[III, Lemma 4.10]{Sa}}]
\label{P:SAKAI_LEM4.10} Let $q\in M\setminus \{\Cut(p)\cup \{p\}\}$ and fix $u\in T_qM$. Take a normal minimal geodesic $\gamma\colon [0,d_p(q)]\to M$ joining $p$ to $q$. Let $X(t)$ be a Jacobi field along $\gamma$ satisfying the boundary condition $X(0)=0$, $X(d_p(q)) = u$, and let 
\[
X^\perp(t)= X(t) - \langle X(t), \dot{\gamma}(t)\rangle\dot{\gamma}(t)
\]
be the Jacobi field that is the vertical component of $X(t)$ with respect to $\dot{\gamma}$. Then
\begin{align}
	\Hess d_p(q)(u,u)=\langle \nabla_{\dot\gamma} X^\perp(d_p(q)),X^\perp(d_p(q))\rangle.
\end{align}
\end{prop}

We now consider the halved square of the distance function. Let $M$ be a complete Riemannian manifold and fix $p\in M$. Let $h_p : M\to \mathbb{R}$ be given by 
\[
q\mapsto h_p(q)= \frac{1}{2}d_p(q)^2.
\]
Then $h_p$ is of class $C^\infty$ on $M\setminus \Cut(p)$ and, for $u\in T_pM$, the following hold:
\begin{align*}
	||\nabla h_p(q)||& = d_p(q),\\ \smallskip
	\Hess h_p(u,u)  & = d_p(q)\langle \nabla_{\dot\gamma} X(d_p(q)),X(d_p(q))\rangle,
\end{align*}
where $X$ and $\gamma$ are as in Proposition~\ref{P:SAKAI_LEM4.10}.

% SS: CONTINUITY OF THE TANGENT CUT LOCUS

\subsection{Stability of the cut locus} Let $M$ be a complete Riemannian manifold. Recall that $\Cut(p)\subset M$ is a closed subset of $M$ for all $p\in M$ (see, for example, \cite[Corollary 13.2.10]{dC92}). Given $p\in M$, we let $B(p,r)$ be an open ball of radius $r$ centered at $p$ and let $\Cut(B(p,r))$ be the cut locus of the ball (defined as the union of the cut loci of the points in $B(p,r)$). Recall that $M$ is \emph{closed} (resp. \emph{open}) if it is compact (resp. non-compact) and has no boundary.  

% DEF: CONDITION CLC1 (CUT LOCUS CONTINUITY)

\begin{definition}[Topological stability]
	\label{D:CLC1}
	Let $M$ be a complete Riemannian manifold. The cut locus $\Cut(p)$ of a point $p\in M$ is \emph{topologically stable} if, for any open neighborhood $U(\Cut(p))$ of $\Cut(p)$ in $M$, there exists $r>0$ such that $\Cut(B(p,r))\subseteq U(\Cut(p))$. We say that  \emph{the cut locus is topologically stable} if it is topologically stable for all $p\in M$.
\end{definition}

The topological stability of the cut locus of a closed Riemannian manifold is a consequence of the following continuity result.

% THM

\begin{thm}
\label{T:CUT_LOCUS_CONTINUITY}
Let $M$ be a closed $C^\infty$-manifold and let $\{g_n\}_{n=1}^\infty$ be a sequence of $C^\infty$ Riemannian metrics on $M$ that converges to a Riemannian metric $g$ in the $C^\infty$ topology. Let $\{(p_n,v_n)\}_{n=1}^\infty\subseteq TM$ be a sequence with $v_n\in \widetilde{\Cut}(p_n)$, the tangent cut locus of $p_n$ with respect to the metric $g_n$. If $(p_n,v_n)\to (p,v)$, for some $(p,v)\in TM$, then $v\in \widetilde{\Cut}(p)$.
\end{thm}

% PROOF

\begin{proof}
Let $\exp_n$ and $\exp$ denote, respectively, the exponential map of $(M,g_n)$ and of $(M,g)$. Since $g_n\to g$ in $C^\infty$ and  the geodesic equations involve only first order derivatives of the metrics,  $\exp_n \to \exp$ in $C^\infty$. Let $p_n\in (M,g_n)$, $v_n\in \widetilde{\Cut}(p_n)\subseteq T_{p_n}M$, and suppose that $(p_n,v_n)\to (p,v)$, with $p\in M$ and $v\in T_pM$. Then, after maybe passing to a subsequence, one of the following holds:
\\
	\begin{itemize}
		\item[(a)] $\det(d((\exp_{n})_{p_n})_{v_n}) = 0$ and this passes to the limit, since the sequence $\{(\exp_{n})_{p_n}\}_{n=1}^\infty$ converges to $\exp_p$ in $C^\infty$, or\\
		
		\item[(b)] for each $n$ there exists $w_n\in T_{p_n}M$ with $w_n\neq v_n$ and $|w_n|=|v_n|$ such that $(\exp_{n})_{p_n}(v_n) = (\exp_{n})_{p_n}(w_n)$, the geodesics $t\mapsto (\exp_{n})_{p_n}(tv_n)$ and $s\mapsto (\exp_{n})_{p_n}(sw_n)$, $0\leq s,t\leq |v_n|$  minimize and $w_n\to w$ for some $w\in T_pM$. If $v=w$, then $\det (d(\exp)_v)=0$. Otherwise, $(\exp_{n}){p_n}$ would be bijective on a neighborhood of $(p,v)$, contradicting the existence of the $p_n$, $v_n$ and $w_n$. If $v\neq w$, we are done as well, so $v$ must lie in the tangent cut locus of $p$ with respect to $g$.\\
	\end{itemize}

	Thus, we conclude in both cases that $v\in \widetilde{\Cut}(p)$.
\end{proof}

% COROLLARY

\begin{cor}
\label{COR:CONT_CUT_LOCUS}
If $M$ is a closed Riemannian manifold, then the cut locus is topologically stable. 
\end{cor}

% PROOF

\begin{proof}
	Suppose, for the sake of contradiction, that the conclusion of the corollary does not hold. Then there exist a point $p\in M$, an open neighborhood $U(\Cut(p))$  and a sequence of points $p_n\in M$ and tangent vectors $v_n\in \widetilde{\Cut}(p_n)$ such that $p_n\to p$ and $\exp_{p_n}(v_n)\notin U(\Cut(p))$. Since $M$ is compact, its diameter is finite, so the sequence $\{(p_n,v_n)\}$ lies in a compact subset of the tangent bundle $TM$. Hence, after passing to a subsequence if necessary, there exists $v\in T_pM$ such that $(p_n,v_n)\to (p,v)$. Thus, by Theorem~\ref{T:CUT_LOCUS_CONTINUITY}, $v\in \widetilde{\Cut}(p)$. Since the exponential map is continuous, $\exp_{p_n}(v_n) \to \exp_p(v)\in \Cut(p)$, which yields a contradiction.  
\end{proof}

Compactness is necessary in Corollary~\ref{COR:CONT_CUT_LOCUS}. Indeed, the cut locus may not be topologically stable for open (i.e. complete and non-compact) manifolds, as the following examples show. 

% EXAMPLE: CYLINDER

\begin{example}[The flat cylinder]
\label{EX:CYLINDER_CLC}In a flat cylinder $C=\sphere^1\times\RR$, the cut locus $\Cut(p)$ of a point $p\in C$ is a line $\ell$ opposite to the line passing through $p$. Hence, the cut locus of a small open neighborhood around $p$ is isometric to a flat band  $(-\varepsilon,\varepsilon)\times \RR$ with $\{0\}\times\RR = \ell$, for some small $\varepsilon >0$. Hence, any open neighborhood of $\Cut(p)=\ell$ that does not contain a band centered at $\ell$ with fixed width yields a counterexample to the topological stability condition. We explicitly construct such a counterexample below.
     
     Consider $C=\CC/\sim$, with $x+iy = z\sim z' = x'+iy'$ if and only if $(x-x')/2\pi\in \ZZ$ and $y=y'$, equipped with the canonical, quotient geometry. For convenience, we identify $C$ with $\{z=x+iy\in \CC: -\pi \leq x\leq \pi\}$ and obtain a locally flat space which is not of global nonpositive curvature.  Rather, since geodesic segments are either vertical segments or traces on $C$ of straight lines in $\CC$, the set
    \[
    \Cut(iy_0) = \{-\pi + y: y\in \RR\}\,,
    \] is non-empty for any $y_0\in \RR$, and 
         \begin{eqnarray*}
      U(\Cut(iy_0)) &=& \{x + iy\in C: \pi -x < 1/|y| < \pi\mbox{ or }\pi + x < 1/|y| < \pi \}\\
      &&~\cup\, \{x+iy\in C: 0 \neq x \mbox{ and } |y|\pi \leq 1\} 
    \end{eqnarray*}
    is a neighborhood of $\Cut(iy_0)$. This neighborhood has the property that for every vertical line $L_{r}=\{z\in \HH: \Re(z) = \pi -r\}$, $0<r<\pi$, we have $L_r \not\subseteq U(\Cut(iy_0))$. However, since every geodesic ball $B(iy_0,r)$ of radius $0<r<\pi$ about $iy_0$ 
    has the property that  $L_{r/2}\subseteq \Cut\big(B(iy_0,r)\big)$ we have that $ \Cut\big(B(iy_0,r) \not\subseteq U(\Cut(iy_0))$ for every $0<r<\pi$, i.e. the cut locus is not topologically stable.\hfill $\square$
\end{example}

    In Example~\ref{EX:CYLINDER_CLC} above, $\Cut\big(B(iy_0,r)\big) = B(\Cut(iy_0,r))$. Thus the following condition, which is weaker than the topological stability of the cut locus, holds on the flat cylinder.

% DEF: METRIC STABILITY OF THE CUT LOCUS

\begin{definition}[Metric stability]
\label{D:CLC1*}
Let $M$ be a complete Riemannian manifold. Let $p\in M$ and let $B(p,r)$ be an open ball of radius $r$ centered at $p$. Let $\Cut(B(p,r))$ be the cut locus of the ball (defined as the union of the cut loci of the points in $B(p,r)$ and let $B(\Cut(p),r)$ denote the open $r$-neighborhood of $\Cut(p)$. We say that the cut locus $\Cut(p)$ of a point $p\in M$ is \emph{metrically continuous at $p\in M$} if, for any $\varepsilon>0$, there exists $\delta>0$ such that $\Cut(B(p,\delta))\subseteq B(\Cut(p),\varepsilon)$. We say that \emph{the cut locus is metrically stable} if $\Cut(p)$ is metrically stable  for all $p\in M$. 
\end{definition}

Note that topological stability implies metric stability. Example~\ref{EX:CYLINDER_CLC} shows that metric stability does not imply topological stability when the manifold is non-compact.  Assuming compactness, however, metric stability does imply topological stability. Hence, in the compact case, both properties are equivalent. 

% THM: TOPOLOGICAL STABILITY IFF METRIC STABILITY IN THE COMPACT CASE

\begin{prop}
	\label{T:CUT_LOCUS_STABLE}
	Let $M$ be a closed Riemannian manifold. If the cut locus is metrically stable, then it is topologically stable.
\end{prop}

% PROOF

\begin{proof}
Suppose that the cut locus is metrically stable, fix $p\in M$ and let $U(\Cut(p))$ be an open neighborhood of  $\Cut(p)$ in $M$. Since $M$ is compact and $\Cut(p)\subset M$ is a closed subset of $M$, it follows that $\Cut(p)$ is compact. Since $M\setminus U(\Cut(p))$ is also compact and the distance function $\dd:M\times M \to \RR$ is continuous, there exists $\varepsilon>0$ such that $B(\Cut(p),\varepsilon)\subseteq U(\Cut(p))$. Since the cut locus is metrically stable, there exists $\delta>0$ such that $\Cut(B(p,\delta))\subseteq B(\Cut(p),\varepsilon)\subseteq U(\Cut(p))$ and the conclusion follows.
\end{proof}

In view of Proposition~\ref{T:CUT_LOCUS_STABLE}, when $M$ is compact we can simply say that the cut locus is \emph{stable}. There are, however, open Riemannian manifolds on which not even metric stability holds.

% EXAMPLE: BELTRAMI TRUMPET

\begin{example}[The Beltrami trumpet]
    We start with the upper half model of the hyperbolic plane, i.e.\ the complex upper half plane $\HH = \{z=x+iy: x\in \RR, y >0\}$ equipped with the \emph{hyperbolic Riemannian metric}
    \[
    ds^2 = \frac{dx^2+dy^2}{y^2}
    \]
      leading to the \emph{hyperbolic distance}
    \begin{eqnarray}\label{hyperbolic-dist:eq}
    \hspace*{1cm}d_{\HH}(x_1 + iy_1, x_2+iy_2) &=& \arcosh \left(1 + \frac{(x_1-x_2)^2 + (y_1-y_2)^2}{2y_1y_2}\right)\,, 
    \end{eqnarray}
    for $z_1=x_1+iy_1, z_2=x_2+iy_2\in \HH$ (see, for example, \cite[Ch.\ 1]{Katok1992}).
    Further, consider its quotient 
    \[
    T = \{[z]: z\in \HH\},\]
     where
    \[
    [z] = \{z+2k\pi : k \in\ZZ\},
    \]
    equipped with the canonical quotient geometry. The Riemannian manifold $T$ is well known under various names, among others the \emph{Beltrami trumpet} (see \cite[Fig.\ 1.60]{Be} and compare with \cite[Ch.\ 5.2]{Sti}, where this manifold is called the \emph{complete pseudosphere}). 
    For convenience, we identify $T$ with
    $$\{z=x+iy: -\pi \leq x < \pi, y >0\}\,$$
    and obtain the quotient distance for $z,w\in T$ given by
    $$ d(z,w) = \min_{k\in\{-1,0,1\}} d_{\HH}(z,w+2k\pi)\,.$$
    Notably, $T$ inherits from $\HH$ constant negative sectional curvature $-1$, but it is no longer of global nonpositive curvature in the triangle comparison sense, i.e.\ it is not a $\mathrm{CAT}(0)$ space, since it is not simply connected.  Rather, since geodesic segments are either vertical segments or traces on $T$ of arc segments on circles vertical to the real line,  
    $$\emptyset \neq \Cut(iy_0) = \{-\pi + y: y>0\}\,,$$ for any $y_0>0$, and 
    for given $\delta>0$, setting $\alpha = (\cosh \delta -1)/2$, 
    \begin{eqnarray*}U(\Cut(iy_0)) &=& \{x + iy\in \HH:  \alpha y> \pi\}\\
    &&~\cup\, \\
    && \{x+iy\in \HH: 0\leq \pi+x < \alpha y\leq \pi\mbox{ or } 0<\pi - x < \alpha y\leq \pi\}\,,
    \end{eqnarray*}
    is a neighborhood of $\Cut(iy_0)$ which, by construction, contains the neighborhood $B(\Cut(iy_0),\delta)$ with hyperbolic $\delta$-distance to $\Cut(iy_0)$.
    For every $\delta>0$ and for every vertical half line $L_{r}=\{z\in \HH: \Re(z) = \pi -r\}$, $0<r<\pi$, we have $L_r \not\subseteq U_\delta(\Cut(iy_0))$. However, since every hyperbolic ball $B(iy_0,r)$ about $iy_0$ of radius $0<r<\arsinh (\pi/y_0)$ is equal to a suitable Euclidean ball 
    $$ B(iy_0,r) = \{z\in \CC: |iy_0\cosh r -z|^2 < y_0\sinh r\}\,,$$
    we have that  $L_{y_0(\sinh r)/2}\subseteq \Cut\big(B(iy_0,r)\big)$.
    In consequence, $ \Cut\big(B(iy_0,r) \not\subseteq B(\Cut(iy_0),\delta)$ for every $\delta, r >0$, i.e.\ the cut locus is not metrically stable.
    
    \hfill $\square$
\end{example}

% REMARK

\begin{rem}
It is tempting to frame the stability notions in Definitions~\ref{D:CLC1} and \ref{D:CLC1*} in terms of the continuity  of the \emph{cut locus function} $\Cut\colon p\mapsto \Cut(P)$, which maps $p\in M$ to its cut locus $\Cut(p)\subseteq M$ and takes values in the set $\mathscr{C}(M)$ of closed subsets of $M$. Note that $\mathscr{C}(M)$ is itself a subset of $\mathscr{P}(M)$, the power set of $M$. Both $\mathscr{C}(P)$ and $\mathscr{P}(M)$ may be topologized in different ways, e.g.\ one may consider the topology on $\mathscr{C}(M)$ induced by the Hausdorff metric (cf.\ \cite[Ch.\ 7]{BBI}) or, on $\mathscr{P}(M)$, the topology generated by the collection 
\[
\beta=\{\, \mathscr{B}\subset \mathscr{P}(X) : \cup_{B\in \mathscr{B}}B \text{ is an open subset of $X$}\,\}.
\]
Besides some obvious observations, it is not clear at the moment what are the precise relations between the continuity of the cut locus function and the stability notions defined above.
\end{rem}

% SECTION: MAIN THEOREMS
%-----------------------

\section{A Riemannian Central Limit Theorem}
\label{S:RCLT}

In this section we prove Theorem~\ref{T:MAIN_THM}. It will be a corollary to  Bhattacharya and Lin's Omnibus Central Limit Theorem, which we recall below, under the assumption that the topological stability of the cut locus holds. We conclude with Example~\ref{EX:CYLINDER_CONT}, which shows that the hypotheses in \cite[Corollary 2.3]{BL} are not sufficient to derive its conclusions, which correspond to those in Theorem~\ref{T:MAIN_THM}.

% STANDING ASSUMPTIONS
%----------------------------------------

\subsection{The Omnibus Central Limit Theorem} Let $(Q,\dd,\mu)$ be a metric measure space and consider the following conditions (cf.\ \cite[Section 2]{BL}):
\\

\begin{itemize}
	\item[(A1)] The Fr\'echet mean $q_o$ of $\mu$ exists and is unique.\\
	\item[(A2a)] There exists a measurable set $U\subseteq Q$ with $q_o\in U$, an open subset $V\subseteq\mathbb{R}^m$, for some $m\geq 1$, and a homeomorphism
	\[
	\varphi: U\subseteq X\to V\subseteq\mathbb{R}^m.
	\]
	Here $U$ is given the relative topology on $Q$.
	\\
	\item[(A2b)] For $\mu$-a.e. $p\in Q$, the function $h(\cdot,p):V\subseteq \mathbb{R}^n\to \mathbb{R}$ given by
	\[
	v\mapsto h(v,p) = \dd^2(\varphi^{-1}(v),p)
	\]
	is twice continuously differentiable, i.e.\ $C^2$, on $V\subseteq \mathbb{R}^m$.\\
	\item[(A3)] $\Prb(q_o^n\in U)\to 1$ as $n\to \infty$.\\
	\item[(A4)] Let $1\leq j,j'\leq n$, fix $p\in Q$ (where the derivatives exist), and let
	\[
	D_j h_p(v) = \frac{\partial h_p}{\partial x_j}(v)\in \mathbb{R},
	\]
	where $x_j$ denote coordinates in $V\subseteq \mathbb{R}^m$.
	
	We let 
	\[
	D_{j,j'}h_p(v) =D_jD_{j'}h_p(v).
	\] 
	Then
	\begin{align}
	\EE|D_jh_{Y_1}(\varphi(q_o))|^2 & <\infty \\
	\EE|D_jD_{j'}h_{Y_1}(\varphi(q_o))|&<\infty \text{ for $j,j'=1,\ldots,m$.}
	\end{align}
	\\
	
	\item[(A5)] (Locally uniform $L^1$ smoothness of the Hessian). Let $1\leq j,j'\leq m$, fix $p\in Q$ (where the derivatives are defined), and let $\varepsilon>0$. Define
	\[
	u_{j,j'}(\varepsilon,p) = \sup\{|D_{j,j'}h(w,p)-D_{j,j'}h(\varphi(q_o),p)|  :  |w-\varphi(q_o)|<\varepsilon \}.
	\]
	Then
	\begin{align}
		\EE|u_{j,j'}(\varepsilon,Y_1)|\to 0 \text{ as  $\varepsilon\to 0$}. 
	\end{align}
	\\
	\item[(A6)] (Non-singularity of the Hessian) The matrix
	\[
	\Lambda = (\EE D_{j,j'}h(\varphi(q_o),Y_1))_{j,j'=1,\ldots,m}
	\]
	is non-singular.
\end{itemize}

% OMNIBUS CLT THEOREM (BL17 THM 2.2)

\begin{thm}[\protect{Omnibus Central Limit Theorem (OCLT), cf.\ \cite[Theorem 2.2]{BL}}]
	\label{T:OCLT} Let $(Q,\dd,\mu)$ be a metric measure space, let $Y_1,\ldots,Y_n \sim \mu$ be $n\geq 1$ independent random variables and suppose that conditions (A1)--(A6) hold. Let $C$ be the covariance matrix of 
	\[
	\{D_jh(\varphi(q_o),Y_1) : j=1,\ldots,m\}.
	\]
	Then
	\begin{align}
	\label{EQ:CLT_THM}
	\sqrt{n}(\varphi(q^n_o)-\varphi(q_o)) \stackrel{\mathcal{L}}{\longrightarrow} N(0,\Lambda^{-1}C\Lambda^{-1}) \text{ as $n\to \infty$}.
	\end{align}
\end{thm}

% REMARKS ON CONDITIONS (A1)-(A6)

\subsection{The hypotheses of the OCLT in the Riemannian case} In preparation for the proof of Theorem~\ref{T:MAIN_THM}, we make the following observations on conditions (A1)-(A6) above in the case where the metric measure space is a complete Riemannian  $n$-manifold equipped with some probability measure. We will always assume that the Riemannian metric is $C^\infty$. \\
\begin{itemize} 
	\item[(RA1)] Condition (A1) holds for complete, simply connected Riemannian manifolds with non-positive sectional curvature, the so-called \emph{Ha\-da\-mard manifolds}). \\
	
	\item[(RA2a)] Condition (A2a) holds for topological manifolds (hence for Riemannian manifolds) and, more generally, for metric measure spaces that are locally homeomorphic to open subsets of $\mathbb{R}^n$. Note that the condition that $U$ be measurable is redundant. Indeed, via the homeomorphism $\varphi$, the set $U$ is open and hence measurable.\\
	
	\item[(RA2b)] Let $M$ be a complete Riemannian manifold with Fr\'echet mean $q_o$ for the volume measure. The fact that the volume of the cut locus of any point in $M$ follows from the independent work of several authors. Indeed, as pointed out in \cite[Section 2.4]{AG}, by the work of Barden and Le \cite{BL}, Hebda \cite{He}, Itoh and Tanaka \cite{IT}, or Li and Nirenberg \cite{LN}, the Hausdorff dimension of the cut locus of any point in $M$ is an integer and is at most $n-1$. Hence, the cut locus of any point in $M$ has volume zero.
	Notably, if instead of the volume one considers a finite measure $\mu$ on $M$, then, under certain mild assumptions, $\Cut(q_o)$ has $\mu$-measure zero due to Le and Barden's Theorem \cite[Theorem 1]{LB}. For a more detailed discussion of the structure of the cut locus up to a set of Hausdorff codimension three, we refer the reader to the work of Angulo Ardoy and Guijarro \cite{AG}.
	Let $U$ be a normal neighborhood of the Fr\'echet mean $q_o\in M$. If $\Cut(U)$ has measure zero, then condition (A2b) holds. 
	\\
	
	\item[(RA3)] Condition (A3) holds by Ziezold's Strong Law of Large Numbers \cite{Zi}.  
\end{itemize}
 It is not clear under which geometric assumptions conditions (A4)--(A6) hold. 
 Elucidating this matter is a problem for future consideration. With these remarks in hand, we are ready to prove Theorem~\ref{T:MAIN_THM}.

% PROOF OF THEOREM A
%--------------------

\subsection{Proof of Theorem~\ref{T:MAIN_THM}}
The main result is a corollary to the Omnibus Central Limit Theorem (Theorem~\ref{T:OCLT} above). Thus, we only need to verify that the hypotheses of this theorem hold under the assumptions of Theorem~\ref{T:MAIN_THM}. By our remarks in the preceding subsection, the hypotheses in Theorem~\ref{T:MAIN_THM} imply conditions (A1), (A2a) and (A3)--(A6) hold. Therefore, to apply the Omnibus Central Limit Theorem, we need only verify that condition (A2b) holds. 

 Since $M$ is closed, it follows from Corollary~\ref{COR:CONT_CUT_LOCUS} that the cut locus is topologically stable. Thus, our two key hypotheses are the following:
\begin{itemize}
	\item[(B)] \emph{The cut locus is topologically stable.}
	
	\item[(C)] \emph{There exists a neighborhood $W$ of $\Cut(q_o)\subseteq M$ such that $\mu(W)=0$.}
\end{itemize}

We will show that 
\begin{center}
	Conditions (B) and (C) $\Rightarrow$ Condition (A2b).
\end{center} 
Note that condition (C) alone does not imply (A2b) (see Example~\ref{EX:CYLINDER_CONT} below).

% CONDITION C'

 Consider now the modified condition
\begin{itemize}
  \item[(C')] \emph{There exists an open ball $B(q_o,r)$ of radius $r>0$ centered at $q_o$ such that $\mu(\Cut(B(q_0,r))) = 0$.}
\end{itemize}
Then it is clear that (C') follows from conditions (B) and (C). The following lemma concludes the proof of Theorem~\ref{T:MAIN_THM}.

% LEMMA

\begin{lem} 
Condition (C') implies condition (A2b). 
\end{lem}

% PROOF

\begin{proof}
  The function $v\mapsto h(v,p) = \dd^2(\varphi^{-1}(v),p)$ is $C^2$ at $v$, if $\varphi^{-1}(v) \notin \Cut(p)$. For the function to be $C^2$ in $v$ for $\mu$-a.e. $p\in Q$, it is thus sufficient that $\varphi^{-1}(v) \notin \Cut(p)$ for $\mu$-a.e. $p\in Q$. Since $\varphi^{-1}(v) \in \Cut(p)$ if, and only if, $p \in \Cut(\varphi^{-1}(v))$, it is thus sufficient that $p \notin \Cut(\varphi^{-1}(v))$ for $\mu$-a.e. $p\in Q$. Now let $V = B(\varphi(q_o), r')$ such that $\varphi^{-1}(V) \subseteq B(q_o,r)$. Then we have from (C') that $\mu(\Cut(\varphi^{-1}(V))) = 0$ and thus we get for every $v \in V$ that $v\mapsto h(v,p) = \dd^2(\varphi^{-1}(v),p)$ is $C^2$ at $v$.
\end{proof}

%\subsection{Condition (C) is not sufficient.}
The topological stability of the cut locus, corresponding to condition (B) above, along with condition (C), played a key role in the proof of Theorem~\ref{T:MAIN_THM}. In \cite[Corollary 2.3]{BL}, which claims the same conclusions as Theorem~\ref{T:MAIN_THM}, only condition (C) is assumed, so it would seem that
\[
\text{Condition (C) $\Rightarrow$ Condition (A2b).}
\]
This is, however, false and below we give a counterexample. By Corollary~\ref{COR:CONT_CUT_LOCUS}, this counterexample must necessarily be non-compact. It follows that it is not possible to apply the OCLT in the Riemannian case if one only assumes that conditions (A1), (A4)--(A6), and (C) hold. Note, though, that if condition (C) holds, then $\Cut(q_o)$ has measure zero.
    
\begin{cexample}[The flat cylinder continued]
	\label{EX:CYLINDER_CONT}
 	We follow the same notation and definitions as in Example~\ref{EX:CYLINDER_CLC}.
    On the cylinder $C$ we have the canonical quotient distance
    given by
   	\[
   	d(z,w) = \min_{k\in\{-1,0,1\}} |z+2k\pi-w|\quad \mbox{ for }z,w\in C\,.
   	\]
    
    Now, introduce the probability measure $\nu$ on $C$ with density 
    \begin{eqnarray}
    \label{cylinder-density:eq}
    f(z) &=& \frac{\gamma}{4y^5 |z|^2}
    \end{eqnarray} 
    for $z=x+iy$ and $\gamma >0$ appropriately chosen, on 
    \[
    R:=\{z+iy\in T:  x = \pm (1/y-\pi) \mbox{ and }y >1\}
    \]
    with respect to the canonical quotient measure on $R$ restricted to the imaginary part only, denoted by $dy$, i.e. 
    \[
    \gamma \int_1^\infty \left(f\left(\pi - \frac{1}{y}+iy\right) + f\left(\frac{1}{y} - \pi+iy\right)\right)\,dy = 1\,.
    \]
    By construction the Euclidean mean on $\CC$ exists for $\nu$, which is then unique and lies, by symmetry, on the imaginary axis and we denote it by $iy_\nu$. Further, we introduce the unit point mass $\delta = \delta_{iy_\nu}$. In consequence, for every $0\leq \alpha\leq 1$ the combined probability measure  $\mu_\alpha = \alpha\nu + (1-\alpha)\delta$ has the unique Fr\'echet mean $iy_\nu$ in $\CC$.
    
    In order to compute the Fr\'echet means of $\mu_\alpha$ on $C$, we consider the corresponding Fr\'echet functions for $w\in C$,
    \begin{eqnarray*}
    F_\alpha(w) &=& \int_C d(z,w)^2\,d\mu_\alpha(z) = \alpha F_1(w) + (1-\alpha) F_0(w),
	\end{eqnarray*}
	where
    \begin{eqnarray*}
    F_1(w) &=& \int_R d(z,w)^2\,d\nu(p), \\
    F_0(w) &=& d(iy_\nu,w)^2~=|iy_\nu-w|^2\,.
    \end{eqnarray*}
    For the following, fix a Euclidean ball $B_r$ about $iy_\nu$ of radius $0<r<\pi$. Then we have
    \[
    \inf_{w\in C\setminus B_r}F_0(w) = r^2\,.
    \]
    On the other hand, letting  $A_r=\sup_{w\in B_r}F_1(w)\geq 0$ and $0\leq \alpha < r^2/(r^2+ A_r)$, for every $w\in T\setminus B_r$ we have
    \begin{eqnarray*}
    F_\alpha(w) & \geq & \alpha F_1(w) + (1-\alpha)r^2\\
     			& \geq & (1-\alpha) r^2\\
     			& > & \alpha A_r\\ 
     			& \geq & F_\alpha(iy_\nu)\,.
    \end{eqnarray*}
    In consequence, whenever $\alpha < r^2/(r^2+ A_r)$, all Fr\'echet means of $\mu_\alpha$ on $C$ lie within $B_r$, and by symmetry, they have the form $w_\alpha = \pm \varepsilon_\alpha + iv_\alpha$, for suitable $0\leq \varepsilon_\alpha<r$ and  $y_\nu -r < v_\alpha < y_\nu + r$. 
    
    We now show that $w_\alpha = iy_\nu$ for $\alpha>0$ sufficiently small. Then (A1) holds, and, as anticipated, (C) holds but not (A2b). 
    To this end, for given $\varepsilon >0$, decompose $R$ into the sets
    \begin{eqnarray*}
     R_\varepsilon := \{z=x+iy\in R: y < 1/\varepsilon\},\\
     \tilde{R}_\varepsilon := \{z=x+iy\in R: y \geq 1/\varepsilon\}\,,
    \end{eqnarray*}
    such that for $w = \varepsilon + iv$,
       \begin{eqnarray*}
	F_1(w) &=& F^\CC_{\nu_{\varepsilon}}(w) + F_{\tilde \nu_{\varepsilon}}(w) \mbox{ where }\\  F^\CC_{\nu_{\varepsilon}}(w) &=& \int_{R_\varepsilon} |w-z|^2 \,f(z)\,dy\\
	F_{\tilde \nu_{\varepsilon}}(w) &=&\int_{\tilde R_\varepsilon} d(w,z)^2 \,f(z)\,dy\,.
	\end{eqnarray*}
	For $w\in B_r$ we also consider 
	\[
	 F^\CC_{\tilde \nu_{\varepsilon}}(w) ~=~\int_{\tilde R_\varepsilon} |w-z|^2 \,f(z)\,dy 
	 \]
	and infer, recalling that $R$ lies above the horizontal line $y=1$, 
       \begin{eqnarray}\nonumber\label{cylinder1:ineq}
	F^\CC_{\tilde \nu_{\varepsilon}}(w)
	&=& \gamma \int_{\frac{1}{\varepsilon}}^\infty \left(
	 \frac{|\pi-1/y+iy-w|^2}{|\pi-1/y+iy|^2} + \frac{|-\pi+1/y+iy-w|^2}{|-\pi+1/y+iy|^2}
	\right) \,\frac{dy}{4y^5}\\
	&\leq& 2\gamma\,\left(1 + \frac{|w|^2}{(\pi-1)^2} \right)\,\varepsilon^4 ~\leq~ 2\gamma \,\left(1 + \frac{|y_\nu|^2 +r^2}{(\pi-1)^2} \right)\,\varepsilon^4
	\,.
	\end{eqnarray}
    
    Finally, choose
    \[
    0 < \alpha < \min\left\{\frac{r^2}{r^2+A_r},\left(1+2\gamma \left(1+\frac{|y_\nu|^2 +r^2}{(\pi-1)^2}\right)    \right)^{-1}\right\}\,,
    \]
    and suppose that $w_\alpha = \varepsilon_\alpha +iv_\alpha \in B_r$ is a Fr\'echet mean of $\mu_\alpha$ which is not $iy_\nu$. If $\varepsilon_\alpha =0$ and $v_\alpha \neq y_\nu$, then we obtain the contradiction,
    \[
    F_\alpha(w_\alpha) = \int_T |w_\alpha-z|^2\,d\mu_\alpha(z) > \int_T |iy_\nu-z|^2\,d\mu_\alpha(z) = F_\alpha(iy_\nu)\,,
    \]
    because $iy_\nu$ is the unique Fr\'echet mean of $\mu_\alpha$ in $\CC$. 
    Else, if $\varepsilon=\varepsilon_\alpha \neq 0$, using (\ref{cylinder1:ineq}), we obtain another contradiction
       \begin{eqnarray*}
     F_\alpha(w_\alpha) &=& (1-\alpha) F_0(w_\alpha) + \alpha  F_{\nu_{\varepsilon}}(w_\alpha) + \alpha F_{\tilde \nu_{\varepsilon}}(w_\alpha)\\
     &=&(1-\alpha) |iy_\nu-w_\alpha|^2 + \alpha  F^\CC_{\nu_{\varepsilon}}(w_\alpha) + \alpha F_{\tilde \nu_{\varepsilon}}(w_\alpha)\\
     &\geq& (1-\alpha) |iy_\nu-w_\alpha|^2 + \alpha  F^\CC_{\nu_{\varepsilon}}(w_\alpha) \\
     &\geq& (1-\alpha)\varepsilon^2 + \alpha  F^\CC_{1}(w_\alpha) - 2\alpha \gamma \left(1+\frac{|y_\nu|^2 +r^2}{(\pi-1)^2}\right)\varepsilon^4   \\
     &> & F^\CC_\alpha(iy_\nu) + \,(1-\alpha)\,(\varepsilon^2- \varepsilon^4)~>~ F^\CC_\alpha(iy_\nu) ~=~F_\alpha(iy_\nu) \,. 
    \end{eqnarray*}
    \hfill $\square$
    \end{cexample}
       \begin{rem}
   At this point we anticipate that if we had used $y^3$ instead of $y^5$ in (\ref{cylinder-density:eq}) we might have encountered smeariness as discussed in \cite{HH15} and \cite{EltznerHuckemann18}. This would, however, violate condition (A6).
   \end{rem}
%\textbf{\textcolor{red}{(We can add a similar example on the Beltrami trumpet which requires slightly more elaborate calculations. Is there anything new we learn?)}}

% ----------------------------------------------------------------
% BEGIN BIBLIOGRAPHY----------------------------------------------
% ----------------------------------------------------------------
\bibliographystyle{amsplain}

% ----------------------------------------------------------------
% END BIBLIOGRAPHY-------------------------------------------------
% ----------------------------------------------------------------

\end{document}